\documentclass[11pt]{article}
\usepackage{times}
\usepackage[english]{babel}
\usepackage[dvips]{graphicx}
\usepackage{amssymb}
\usepackage{amsfonts}
\usepackage{amsmath}
\usepackage{mathrsfs}
\usepackage{url} 
\usepackage{setspace}
\usepackage[numbers,sort&compress,longnamesfirst,sectionbib]{natbib}             
\usepackage{url}
\usepackage{graphicx}
\usepackage{epsfig}

\graphicspath{{./figs/}}

\newcommand{\qed}{\hfill \mbox{\raggedright $\Box$}}

\DeclareSymbolFont{AMSb}{U}{msb}{m}{n}
\newcommand{\N}{{\mathbb{N}}}
\newcommand{\n}{\tfrac{1}{n}}

\def\mod{\operatorname{mod}}

\renewcommand{\leq}{\leqslant}
\renewcommand{\geq}{\geqslant}

\newtheorem{theorem}{Theorem}

\newtheorem{lemma}{Lemma}

\newtheorem{proposition}{Proposition}
\newtheorem{definition}{Definition}

\newenvironment{proof}{\noindent{\bf Proof}\hspace*{1em}}{\qed\bigskip}
\newenvironment{proof-sketch}{\noindent{\bf Sketch of Proof}\hspace*{1em}}{\qed\bigskip}
\newenvironment{proof-idea}{\noindent{\bf Proof Idea}\hspace*{1em}}{\qed\bigskip}
\newenvironment{proof-of-lemma}[1]{\noindent{\bf Proof of Lemma #1}\hspace*{1em}}{\qed\bigskip}
\newenvironment{proof-attempt}{\noindent{\bf Proof Attempt}\hspace*{1em}}{\qed\bigskip}

\newenvironment{remark}{\noindent{\bf Remark}\hspace*{1em}}{\bigskip}

\makeatletter
\def\fnum@figure{{\bf Figure \thefigure}}
\def\fnum@table{{\bf Table \thetable}}
\long\def\@mycaption#1[#2]#3{\addcontentsline{\csname
  ext@#1\endcsname}{#1}{\protect\numberline{\csname
  the#1\endcsname}{\ignorespaces #2}}\par
  \begingroup
    \@parboxrestore
    \small
    \@makecaption{\csname fnum@#1\endcsname}{\ignorespaces #3}\par
  \endgroup}
\def\mycaption{\refstepcounter\@captype \@dblarg{\@mycaption\@captype}}
\makeatother

\def\pr{\mathbb{P}}

\def\bin{{\rm Bin}}

\def\th{\theta}

\def\beq{\begin{equation}}
\def\eeq{\end{equation}}
\def\beqn{\begin{eqnarray}}
\def\eeqn{\end{eqnarray}}

\def\tab{\mathbb{T}_{a,b}}
\def\vtab{\vec{\mathbb{T}}_{a,b}}

\def\n{\eta}
\def\vn{\vec{\eta}}

\def\vp{\vec{p}}

\def\z{\zeta}
\def\vz{\vec{\zeta}}
\def\zin{\zeta_{\infty}}
\def\vzin{\vec{\zeta}_{\infty}}

\def\vx{\vec{x}}
\def\vy{\vec{y}}

\def\vs{\vec{s}}
\def\vS{\vec{S}}

\def\vp{\vec{p}}

\title{Bootstrap Percolation on Periodic Trees}
\author{
Milan Bradonji\'c and Iraj Saniee\\
\\
Mathematics of Networks and Systems\\ 
Bell Labs, Alcatel-Lucent\\ 
600 Mountain Avenue, Murray Hill, NJ 07974, USA 
\\
\texttt{\{milan,iis\}@research.bell-labs.com}
}

\begin{document}
\maketitle

\begin{abstract}
We study bootstrap percolation with the threshold parameter $\theta \geq 2$ and the initial probability $p$ on infinite periodic trees that are defined as follows. Each node of a tree has degree selected from a finite predefined set of non-negative integers and starting from any node, all nodes at the same graph distance from it have the same degree. We show the existence of the critical threshold $p_f(\theta) \in (0,1)$ such that with high probability, (i) if $p > p_f(\theta)$ then the periodic tree becomes fully active, while (ii) if $p < p_f(\theta)$ then a periodic tree does not become fully active. We also derive a system of recurrence equations for the critical threshold $p_f(\theta)$ and compute these numerically for a collection of periodic trees and various values of $\theta$, thus extending previous results for regular (homogeneous) trees.
\end{abstract}

\section{Introduction}
\label{sec:intro}
Bootstrap percolation is a dynamic growth model generalizing cellular 
automata from square grids to arbitrary graphs. Starting from a random 
distribution of some contagious characteristic or feature over the nodes 
of a network (often infinite), new nodes iteratively may acquire the 
feature based on the density of nodes possessing it in their immediate
neighborhoods. The goal is to determine under what conditions the contagion or feature 
spreads over almost all the nodes. 
In particular, there may exist a probability $p_f$, the percolation threshold, 
which characterizes the initial distribution of the said feature, and  gives rise to almost full contagion. 
Clearly such a threshold will depend on the structure of the network and the local activation rule characterized by
parameter $\theta \geq 2$ which determines when an inactive node becomes 
active. 
Bootstrap percolation is therefore a useful model to study spread of viruses between
communities, diffusion of attacks on the web or 
growth of the so-called ``viral content'' in social networks.  It turns out that other 
than for regular trees, Euclidean lattices and some random graph models, there 
are no analytical results on the percolation threshold~\cite{chalupa-1979-bootstrap,balogh-2006-bootstrap,aizenman-1988-metastability,holroyd-2003-sharp,balogh-2012-sharp,balogh-2007-bootstrap,amini-2010-bootstrap,bradonjic-saniee-2013}.

In this work, we study both analytically and numerically bootstrap percolation on periodic trees.  
Periodic trees are useful for estimating upper bounds on the percolation thresholds of various types of semi-regular Euclidean and 
non-Euclidean lattices. Generally, trees play an important role in estimating or bounding percolation threshold for complex graphs. For example, the percolation threshold of a spanning tree of a graph is an upper bound on the percolation threshold of the graph. Thus, when the spanning tree is regular, as in Figure~\ref{fig:tree}, existing results can be used~\cite{chalupa-1979-bootstrap,fontes-2008-bootstrap}. 
Our goal is to extend those results further through derivation of exact thresholds for 
\emph{periodic trees} in which: (i) nodal degrees form a finite set of non-negative integers, 
and (ii) nodes at the same graph distance from any given node 
have the same degree, see Figure~\ref{fig:vt43} (left). We make 
these definitions more precise in Section~\ref{sec:bootstrap}. 
To this end, we derive explicit equations for the percolation threshold for periodic trees as function of
the degree sequence and $\theta$, the threshold parameter. To illustrate, we compute the percolation threshold for several periodic trees.  

Prior work on bootstrap percolation on trees includes the original paper of Chalupa~{\it et al}~\cite{chalupa-1979-bootstrap} which 
introduced bootstrap percolation (on regular trees) and obtained a fundamental recursion for computation of the critical threshold. More recently Balogh~{\it et al}~\cite{balogh-2006-bootstrap} obtained new results for non-regular (infinite) trees. Our approach leverages techniques introduced by Fontes and Schonmann~\cite{fontes-2008-bootstrap}, who derived percolation thresholds for the infinite cluster as well as almost sure activation of bootstrap percolation on regular trees.

\begin{figure}[ht]
\begin{minipage}[b]{0.495\linewidth}
\centering
\includegraphics[width = 2.5in]{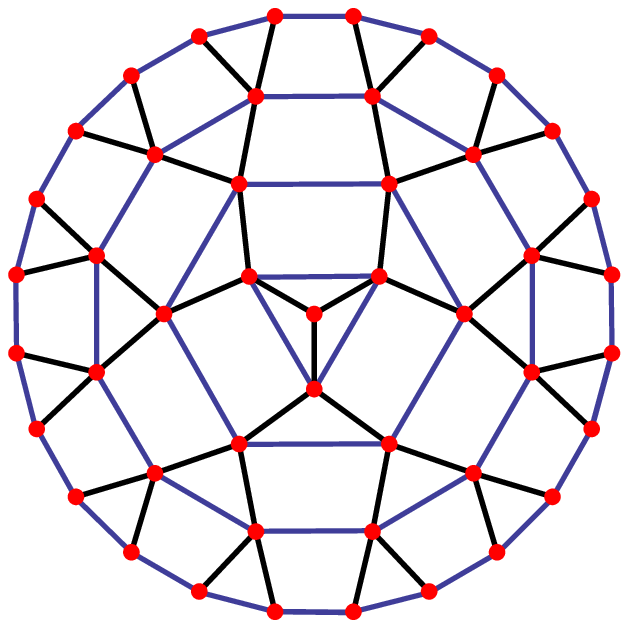}
\label{fig:figure1}
\end{minipage}
\hspace{0.15cm}
\begin{minipage}[b]{0.495\linewidth}
\centering
\includegraphics[width = 2.5in]{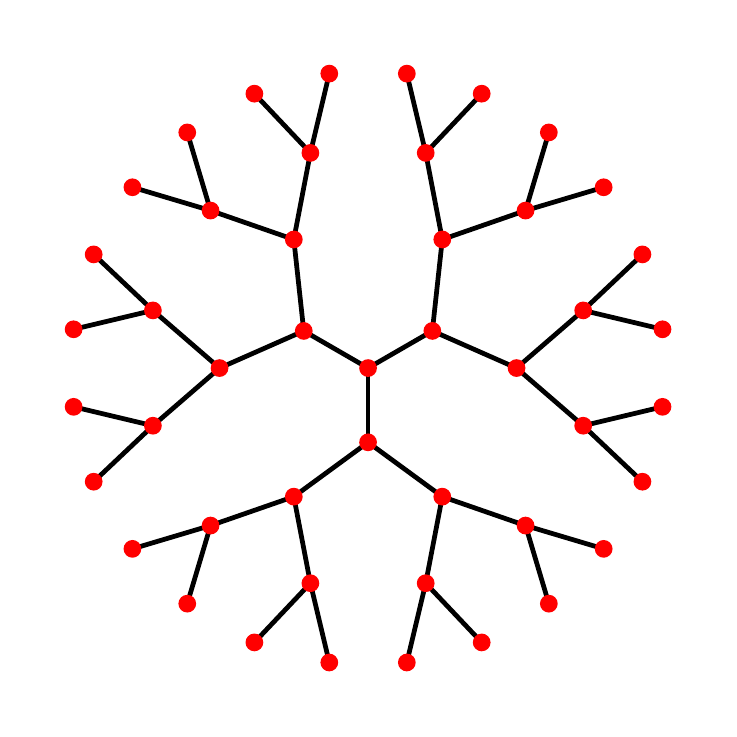}
\label{fig:figure2}
\end{minipage}
\caption{A regular spanning tree to approximate the percolation threshold of a graph.}
\label{fig:tree}
\end{figure}

\section{Bootstrap percolation process on periodic trees}
\label{sec:bootstrap}

Bootstrap percolation (BP) is a cellular automaton defined on an underlying graph $G=(V,E)$ with state space $\{0,1\}^{V}$ whose initial configuration is chosen by a Bernoulli product measure. In other words, every node is in one of two different states $0$ or $1$, {\it inactive} or {\it active} respectively, and a node is active with probability $p$, independently of other nodes,s within the initial configuration. 

After drawing an initial configuration at time $t=0$, a discrete time deterministic process updates the configuration according to a local rule:
an inactive node becomes active at time $t+1$ if the number of its active neighbors at $t$ (in the sense of graph distance) is greater than or equal to some specified {\it threshold parameter} $\theta$. Once an inactive node becomes active it remains active forever. A configuration that does not change at the next time step is a {\it stable} configuration.
A configuration is {\it fully active} if all its nodes of are active.

An interesting phenomenon to study is metastability near a first-order phase transition: 
Does there exist $0< p_c <1$ such that: 
\beq
\nonumber
\left(\forall p<p_c\right) \lim_{t \to \infty} \pr_p \left(V \textrm{ becomes fully active}\right) = 0\,,
\eeq
and 
\beq
\nonumber
\left(\forall p>p_c\right) \lim_{t \to \infty} \pr_p \left(V \textrm{ becomes fully active}\right) =1 \,?
\eeq

In this work we study bootstrap percolation processes and 
associated $p_c$'s on periodic trees defined as follows. 

\begin{definition}
\label{def:period.tree}(Periodic Tree)
Let $\ell, m_0,m_1,\dots,m_{\ell-1} \in \N$. An $\ell$-periodic tree $\mathbb{T}_{m_0,m_1,\dots,m_{\ell-1}}$ is recursively defined as follows.
Consider a node $\emptyset$, called root. The nodes at the distance $k \mod \ell$ from $\emptyset$ have degree $m_k+1$ for $k \in \N$.
\end{definition}
A regular (ordinary) tree $T_d$ is a 1-periodic tree 
where each node has degree $d+1$.

The schematic presentation of $T_{3,2}$ is given in Figure~\ref{fig:vt43}.
Notice that nodes in this tree have degrees 4 and 3; those at even 
distance from the node in the center have degree 4 and those at odd distance have 
degree 3.

\begin{figure}[!h]
\begin{centering}
\setlength{\unitlength}{0.505mm}
\begin{picture}(150,100)
\thicklines
\put(50,50){\line(1,0){20}}
\put(50,50){\line(0,1){20}}
\put(50,50){\line(-1,0){20}}
\put(50,50){\line(0,-1){20}}
\put(70,50){\line(1,1){15}}
\put(70,50){\line(1,-1){15}}
\put(30,50){\line(-1,1){15}}
\put(30,50){\line(-1,-1){15}}
\put(50,70){\line(1,1){15}}
\put(50,70){\line(-1,1){15}}
\put(50,30){\line(1,-1){15}}
\put(50,30){\line(-1,-1){15}}
\put(85,65){\line(1,-1){7}}
\put(85,65){\line(1,1){7}}
\put(85,65){\line(-1,1){7}}
\put(65,85){\line(-1,1){7}}
\put(65,85){\line(1,-1){7}}
\put(65,85){\line(1,1){7}}
\put(85,35){\line(-1,-1){7}}
\put(85,35){\line(1,-1){7}}
\put(85,35){\line(1,1){7}}
\put(35,85){\line(-1,-1){7}}
\put(35,85){\line(-1,1){7}}
\put(35,85){\line(1,1){7}}
\put(15,65){\line(-1,-1){7}}
\put(15,65){\line(-1,1){7}}
\put(15,65){\line(1,1){7}}
\put(15,35){\line(-1,-1){7}}
\put(15,35){\line(-1,1){7}}
\put(15,35){\line(1,-1){7}}
\put(35,15){\line(-1,-1){7}}
\put(35,15){\line(-1,1){7}}
\put(35,15){\line(1,-1){7}}
\put(65,15){\line(-1,-1){7}}
\put(65,15){\line(1,1){7}}
\put(65,15){\line(1,-1){7}}

\put(50,50){\circle*{2}}
\put(70,50){\circle{2}}
\put(30,50){\circle{2}}
\put(50,70){\circle{2}}
\put(50,30){\circle{2}}
\put(15,65){\circle*{2}}
\put(15,35){\circle*{2}}
\put(85,65){\circle*{2}}
\put(85,35){\circle*{2}}
\put(65,85){\circle*{2}}
\put(65,15){\circle*{2}}
\put(35,85){\circle*{2}}
\put(35,15){\circle*{2}}
\put(92,72){\circle{2}}
\put(92,28){\circle{2}}
\put(8,72){\circle{2}}
\put(8,28){\circle{2}}
\put(92,58){\circle{2}}
\put(92,42){\circle{2}}
\put(8,58){\circle{2}}
\put(8,42){\circle{2}}
\put(78,72){\circle{2}}
\put(78,28){\circle{2}}
\put(22,72){\circle{2}}
\put(22,28){\circle{2}}
\put(72,92){\circle{2}}
\put(58,92){\circle{2}}
\put(72,78){\circle{2}}
\put(72,22){\circle{2}}
\put(72,8){\circle{2}}
\put(58,8){\circle{2}}
\put(42,92){\circle{2}}
\put(28,92){\circle{2}}
\put(28,78){\circle{2}}
\put(42,8){\circle{2}}
\put(28,22){\circle{2}}
\put(28,8){\circle{2}}
\end{picture}
\setlength{\unitlength}{0.505mm}
\begin{picture}(150,100)
\thicklines
\put(70,50){\vector(-1,0){20}}
\put(50,50){\vector(0,1){20}}
\put(30,50){\vector(+1,0){20}}
\put(50,30){\vector(0,+1){20}}
\put(85,65){\vector(-1,-1){15}}
\put(85,35){\vector(-1,1){15}}
\put(15,65){\vector(1,-1){15}}
\put(15,35){\vector(+1,+1){15}}
\put(50,70){\vector(1,1){15}}
\put(35,85){\vector(+1,-1){15}}
\put(65,15){\vector(-1,+1){15}}
\put(35,15){\vector(+1,+1){15}}
\put(92,58){\vector(-1,+1){7}}
\put(92,72){\vector(-1,-1){7}}
\put(78,72){\vector(+1,-1){7}}
\put(58,92){\vector(+1,-1){7}}
\put(72,78){\vector(-1,+1){7}}
\put(65,85){\vector(1,1){7}}
\put(78,28){\vector(+1,+1){7}}
\put(92,28){\vector(-1,+1){7}}
\put(92,42){\vector(-1,-1){7}}
\put(28,78){\vector(+1,+1){7}}
\put(28,92){\vector(+1,-1){7}}
\put(42,92){\vector(-1,-1){7}}
\put(8,58){\vector(+1,+1){7}}
\put(8,72){\vector(+1,-1){7}}
\put(22,72){\vector(-1,-1){7}}
\put(8,28){\vector(+1,+1){7}}
\put(8,42){\vector(+1,-1){7}}
\put(22,28){\vector(-1,+1){7}}
\put(28,8){\vector(+1,+1){7}}
\put(28,22){\vector(+1,-1){7}}
\put(42,8){\vector(-1,+1){7}}
\put(58,8){\vector(+1,+1){7}}
\put(72,22){\vector(-1,-1){7}}
\put(72,8){\vector(-1,+1){7}}

\put(50,50){\circle*{2}}
\put(70,50){\circle{2}}
\put(30,50){\circle{2}}
\put(50,70){\circle{2}}
\put(50,30){\circle{2}}
\put(15,65){\circle*{2}}
\put(15,35){\circle*{2}}
\put(85,65){\circle*{2}}
\put(85,35){\circle*{2}}
\put(65,85){\circle*{2}}
\put(65,15){\circle*{2}}
\put(35,85){\circle*{2}}
\put(35,15){\circle*{2}}
\put(92,72){\circle{2}}
\put(92,28){\circle{2}}
\put(8,72){\circle{2}}
\put(8,28){\circle{2}}
\put(92,58){\circle{2}}
\put(92,42){\circle{2}}
\put(8,58){\circle{2}}
\put(8,42){\circle{2}}
\put(78,72){\circle{2}}
\put(78,28){\circle{2}}
\put(22,72){\circle{2}}
\put(22,28){\circle{2}}
\put(72,92){\circle{2}}
\put(58,92){\circle{2}}
\put(72,78){\circle{2}}
\put(72,22){\circle{2}}
\put(72,8){\circle{2}}
\put(58,8){\circle{2}}
\put(42,92){\circle{2}}
\put(28,92){\circle{2}}
\put(28,78){\circle{2}}
\put(42,8){\circle{2}}
\put(28,22){\circle{2}}
\put(28,8){\circle{2}}
\end{picture}
\end{centering}
\caption{Periodic tree $T_{3,2}$ (left), and its oriented version $\vec{T}_{3,2}$ (right).}
\label{fig:vt43}
\end{figure}
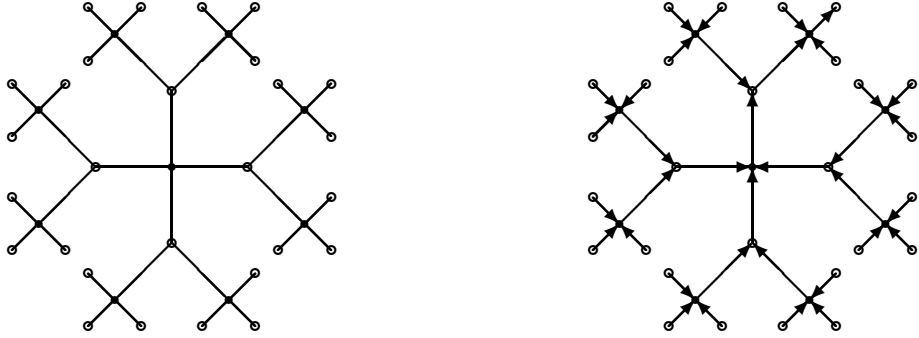

\begin{definition}
\label{def:orient.period.tree}(Oriented $\ell$-Periodic Tree)
Let $\ell, m_0,m_1,\dots,m_{\ell-1} \in \N$. An oriented $\ell$-periodic tree $\vec{\mathbb{T}}_{m_0,m_1,\dots,m_{\ell-1}}$ is recursively defined as follows. Consider a node $\emptyset$, called root. The nodes at the distance $k \mod \ell$ from $\emptyset$ have in-degree $m_k$ and out-degree $1$ for $k \in \N$.
\end{definition}
The schematic presentation of $T_{3,2}$ and its oriented version 
$\vec{T}_{3,2}$ are given in Figure~\ref{fig:vt43}. 
The following Lemma~\ref{lm:charact.pc} is an important ingredient for our main result given by Theorem~\ref{thm:mainresult}, which we prove directly. This result has appeared in different forms in~\cite{fontes-2008-bootstrap,biskup-2009-metastable}.

\begin{lemma}
\label{lm:charact.pc}
Given $n, \th \in \N$ such that $2 \leq \th \leq n-1$ and $x \in [0,1]$ let
\beq
\label{eq:phi}
\phi_{n,p,\th} (x):=p + (1-p)\sum_{k=\th}^n {n \choose k} x^k (1-x)^{n-k} \,.
\eeq
There exists $p_c \in (0,1)$ such that for any $p > p_c$ we have $\phi_{n,p,\th} (x) > x$ for every $x \in (0,1)$, and $1$ is the only solution of $\phi_{n,p,\th} (x) = x$ in $[0,1]$.
\end{lemma}

\begin{proof}
From~(\ref{eq:phi}) it is immediately clear that $x=1$ is a solution of $\phi_{n,p,\th}(x)=x$ in $[0,1]$.
Given $n$ and $\th$, let us define the function $\Phi_{p}(x):=\phi_{n,p,\th}(x) - x$. 
The proof follows from analyzing $\Phi_{p}(x)$ as a function
of $x$ and its continuity and monotonicity as a 
function of $p$. The first and second derivatives, as functions of $x$,
are given by:
\beqn
\label{eq:phiprim}
\nonumber
\Phi'_{p}(x) &=& (1-p)n{n-1 \choose \th-1}x^{\th-1}(1-x)^{n-\th}-1 \,, \\
\nonumber
\Phi''_{p}(x) &=& (1-p)n{n-1 \choose \th-1}x^{\th-2}(1-x)^{n-\th-1} \left( \th -1 - (n-1)x \right) \,.
\eeqn

By considering $\Phi''_{p}(x)$, we see that $x^*:=(\th-1)/(n-1)$ is a unique stationary point of the fist derivative $\Phi'_{p}(x)$ in the open interval $(0,1)$.
Therefore $\Phi'_{p}(x)$ is strictly increasing on $[0,x^*)$, strictly decreasing on $(x^*,1]$, and attains its maximum value at $x^*$ given by
\beq
\label{eq:phi.max.derivative}
\Phi'_p\left( x^* \right) = (1-p)n{n-1 \choose \th-1} \left( \frac{\th-1}{n-1} \right)^{\th-1} \left( \frac{n-\th}{n-1} \right)^{n-\th} - 1 \,.
\eeq

For a given $n$, it is evident from (\ref{eq:phi.max.derivative}) that there exists $p^* \in (0,1)$ such that $\Phi'_{p}(x^*) < 0$.
Hence, for every $p>p^*$, the first derivative is less than zero $\Phi'_p(x)<0$, and the function $\Phi_p(x)$ is strictly decreasing. We have $\Phi_p(0)=p$ and $\Phi_p(1)=0$. 
Therefore $\Phi_p(x)$ is strictly positive in $(0,1)$ and thus $\phi_{n,p,\th}(x)>x$ in $(0,1)$. 

This does not conclude the proof yet, as $p^*$ is not the  $p_c$ in the assertion of the lemma. 
To identify $p_c$, we analyze $\Phi_0(x)=\sum_{k=\th}^n {n \choose k} x^k (1-x)^{n-k}-x$. The idea is to show that $\Phi_0(x)=0$ has a real root in $(0,1)$.
Then the monotonicity and continuity of $\Phi_p(x)$ in $p$ (a linear function of $p$) will lead to the existence of the critical $p_c$ in $(0,1)$ for which $\Phi_{p_c}(x)=0$ has a unique solution in $(0,1)$.

By simple substitution, we have $\Phi_0(0)=\Phi_0(1)=0$ and $\Phi'_0(0)=\Phi'_0(1)=-1$, 
so there exists a root $r \in (0,1)$ such that $\Phi_0(r)=0$. 
We have already shown that for any $p>p^*$, $\Phi_p(x)>0$ for every $x$ in $(0,1)$. 
Observe that $\Phi_{p}(x)$ is strictly increasing and continuous in $p \in [0,1]$. 
Hence there exists $0<p_c<p^*$ such that the equation $\Phi_p(x)=0$ has real root(s) in $(0,1)$ for every $p \leq p_c$, which is given by
\beq
p_c = \inf \left\{ p \in (0,1): \textrm{$\phi_{n,p,\th}(x)>x$ for every $x$ in $(0,1)$} \right\} \,.
\eeq
This concludes the proof. 
\end{proof}

\begin{remark}
The critical value $p_c$ can be computed as the solution for $p$ of the system of two equations $\phi_{n,p,\th}(x)=x$ and $\phi'_{n,p,\th}(x)=1$ (respectively, $\Phi_p(x)=0$ and $\Phi'_p(x)=0$) in $p$ and $x$ in $(0,1)^2$. Concretely, this system is given by:
\beqn
\nonumber
p + (1-p)\sum_{k=\th}^n {n \choose k} x^k (1-x)^{n-k} =x \,,\\
\nonumber
(1-p)n{n-1 \choose \th-1}x^{\th-1}(1-x)^{n-\th}=1 \,.
\eeqn
\end{remark}

\begin{remark}
It is not hard to show, by analyzing $\Phi_p'(x)$ and $\Phi''_p(x)$, that for $p<p_c$ the equation $\phi_{n,p,\th}(x)=x$ has exactly two real solutions in $(0,1)$, see Figure~\ref{fig:curve03}, and no roots when $p>p_c$, see Figure~\ref{fig:curve04}. (The fact that $0.3 < p_c < 0.4$ may be found in Figure~\ref{fig:numerics}, top-left, in the fourth curve from the bottom which corresponds to $a=b=8$ and $\th=5$.)
\end{remark}
 
\begin{figure}[!h]  
\centering 
\includegraphics[width=0.95\columnwidth]{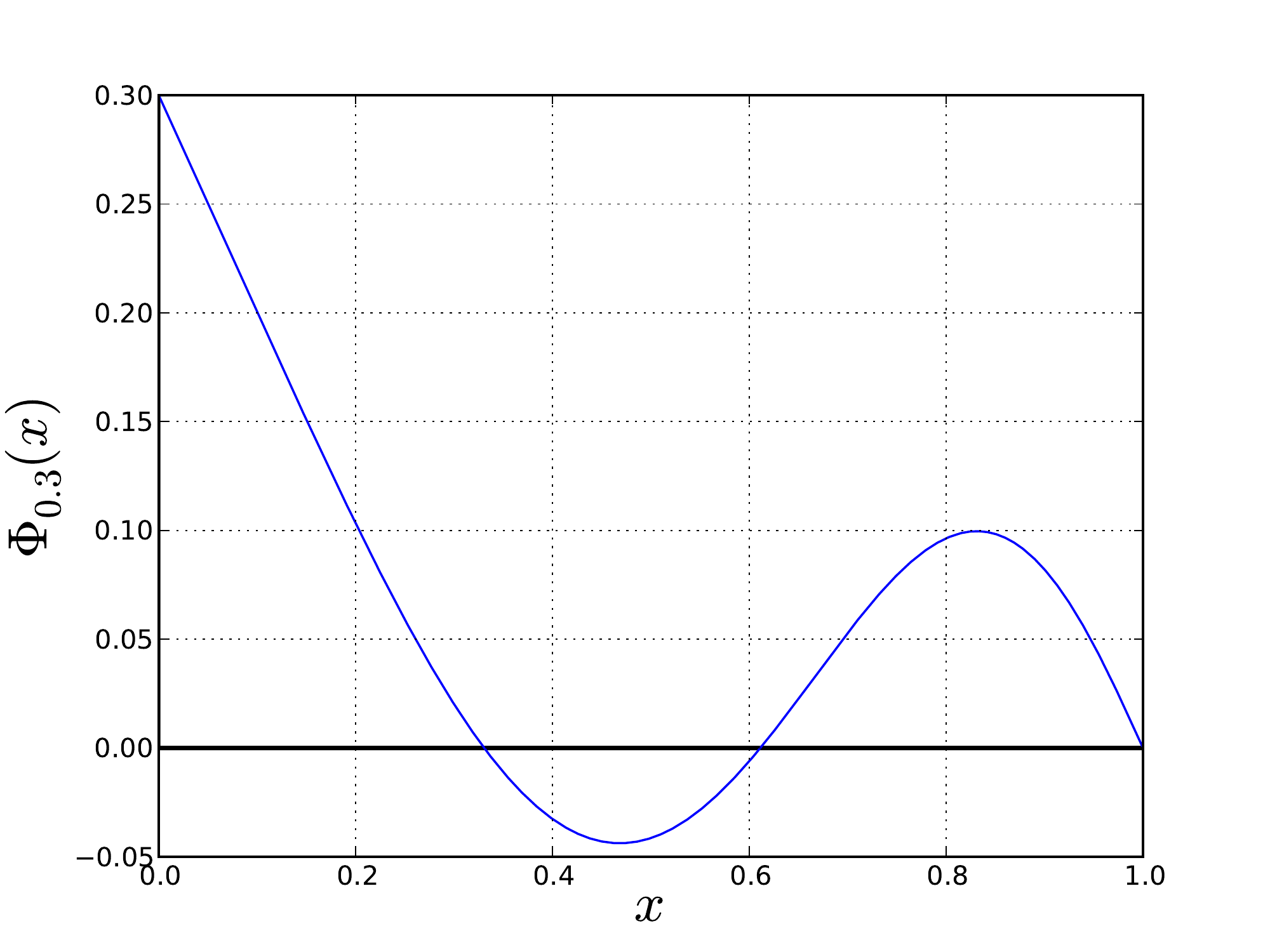}
\caption{$\Phi_{0.3}(x)$ for $n=7, \th=5, p=0.3$.}
\label{fig:curve03}
\end{figure}

\begin{figure}[!h]  
\centering 
\includegraphics[width=0.95\columnwidth]{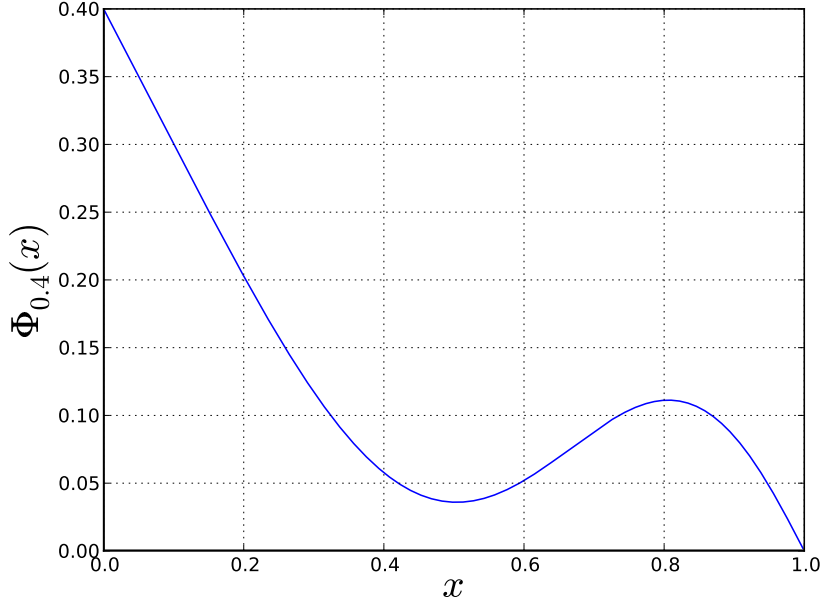}
\caption{$\Phi_{0.4}(x)$ for $n=7, \th=5, p=0.4$.}
\label{fig:curve04}
\end{figure}

\section{Main result}
\label{sec:bp_period_tree}

We give a proof for the case of a tree of periodicity two and then indicate
how the result may be proved for larger periodicity. 

\begin{theorem}
\label{thm:mainresult}
Given $a,b \in \N$ and $2 \leq \th < a,b$ consider a BP on $\tab$ with the initial probability $p$. There exists $p_f \in (0,1)$ such that for all $p \geq p_f$, the tree $\tab$ is fully active a.a.s., and $\tab$ is not fully active a.a.s. for $p<p_f$.
\end{theorem}

To prove Theorem~\ref{thm:mainresult}, we adopt the methodology of~\cite{fontes-2008-bootstrap}. That is, we first derive the percolation
threshold, $\vp_f$, for the oriented periodic tree, $\vtab$ (see Definition~\ref{def:orient.period.tree}), using a system of 
recurrence equations (Theorem~\ref{thm:bp.orient}).  Next, we show that the percolation threshold, $p_f$
for the unoriented periodic tree $\tab$ has to be the same as the threshold for $\vtab$, that is $p_f=\vp_f$ (Theorem~\ref{prop2}). These two theorems complete the proof. 

\subsection{BP on an oriented tree $\vtab$}
\label{sec:bp:orient}

\begin{theorem}
\label{thm:bp.orient}
Given $a,b \in \N$ and $2 \leq \th < a,b$ consider a BP on $\tab$ with the initial probability $p$. There exists $\vp_f \in (0,1)$ such that for all $p > \vp_f$, the tree $\vtab$ is fully active a.a.s., and $\vtab$ is not fully active a.a.s. for $p < \vp_f$.
\end{theorem} 

\begin{proof}
The dynamics of bootstrap percolation process on $\vtab$ is captured by knowing the states of nodes, that is, $\vz_t(v) \in \{0,1\}$ and $\vn_t(u) \in \{0,1\}$ for every $v \in V_a$ and $u \in V_b$ at $t \in \N_0$. 
Denote by $V_a$ and $V_b$ the two sets of nodes of degrees $a+1$ and $b+1$ respectively in $\tab$, that is, the sets of nodes of in-degrees $a$ and $b$ in $\vtab$.

Choose any node $v \in V_a$. Conditioning upon whether this node $v$ was active at time $0$ or not (i.e., $\vn_0 (v)=0$ or $\vn_0 (v)=1$), the probability that the node $v$ is active at time $t$ is given by
\begin{align*}
\pr \left( \vn_{t}(v)=1 \right) &= \pr \left( \vn_{0}(v)=1 \right) + \pr \left( \vn_{0}(v)=0 \right) 
\pr \left( \sum_{u \leadsto v} \vz_t(u) \geq \th : \vn_{0}(v)=0 \right) \,,
\end{align*}
where the symbol ``$\leadsto$'' indicates that $u$ is a neighbor of $v$ in the oriented tree $\vtab$ and the edge orientation is from $u$ to $v$.

Also, choose any node $u \in V_b$, independently of $v$. Analogously, the probability that node $u$ is active at time $t$ is given by
\begin{align*}
\pr \left( \vz_{t}(u)=1 \right) &= \pr \left( \vz_{0}(u)=1 \right) + \pr \left( \vz_{0}(u)=0 \right) 
\pr \left( \sum_{v \leadsto u} \vn_t(v) \geq \th : \vz_{0}(u)=0 \right) \,.
\end{align*}
Given symmetry and dynamical rules of the BP process, $\vz_{t}(x)$ are independent Bernoulli random variables and moreover independent of $\vn_0(v)$. 
Hence letting $\vx_t:=\pr\left(\vn_{t}(v)=1 \right)$ and $\vy_t:=\pr\left(\vz_{t}(u)=1 \right)$, we obtain
\beqn
\label{eq:vx}
\vx_{t} &=& p + (1-p) \sum_{k=\th}^{a} {a \choose k} \vy_{t-1}^k \left(1-\vy_{t-1} \right)^{a-k} \,, \\
\label{eq:vy}
\vy_{t} &=& p + (1-p) \sum_{k=\th}^{b} {b \choose k} \vx_{t-1}^k \left(1-\vx_{t-1} \right)^{b-k} \,,
\eeqn
where $\vx_0=p$ and $\vy_0=p$.

In order to simplify the notation, we define an auxiliary function 
\beq
\phi_{n,p,\th} (x):=p + (1-p)\sum_{k=\th}^n {n \choose k} x^k (1-x)^{n-k}\,,
\eeq
for $x \in [0,1]$, where $n, \th \in \N$ are given, such that $2 \leq \th \leq n-1$. The function $\phi_{n,p,\th}(x)$ is strictly increasing in $x$ (the first derivative in $x$ is positive in $(0,1)$). Moreover, given $x \in [0,1]$, the mapping $p \to \phi_{n,p,\th}(x)$ is strictly increasing in $p$ in $(0,1)$, (the first derivative in $p$ is positive in $(0,1)$).

From the definition of $\phi_{a,p,\th}$ and $\phi_{b,p,\th}$, the recurrence equations~(\ref{eq:vx}) and~(\ref{eq:vy}) can be rewritten in a more compact form
\begin{align}
\label{eq:req.sys.x}
\vx_{t} &= \phi_{a,p,\th}(\vy_{t-1}) \,,\\
\label{eq:req.sys.y}
\vy_{t} &= \phi_{b,p,\th}(\vx_{t-1}) \,.
\end{align}

We now show that the limits $\vx_\infty:=\lim_{t \to \infty} \vx_t$ and $\vy_\infty:=\lim_{t \to \infty} \vy_t$ exist. First, we show that the sequences $\{\vx_{t}\}_{t=0}^{\infty}$ and $\{\vy_{t}\}_{t=0}^{\infty}$ are increasing in $t$. By definition $\vx_{0}=p$ and $\vy_{0}=p$.
The monotonicity of $\phi_{a,p,\th}$ and $\phi_{a,p,\th}$, and~(\ref{eq:req.sys.x}) and~(\ref{eq:req.sys.y}) yield $\vx_{1} = \phi_{a,p,\th}(\vy_{0}) \geq \vy_0=p$ and similarly $\vy_{1} = \phi_{b,p,\th}(\vx_{0}) \geq \vx_0=p$. Hence $\vx_1 \geq \vx_0$ and $\vy_1 \geq \vy_0$. Assume that $\vx_{t}  \geq \vx_{t-1}$ and $\vy_{t}  \geq \vy_{t-1}$ for some $t \geq 1$.
Then it follows $\vx_{t+1} = \phi_{a,p,\th}(\vy_t) \geq \phi_{a,p,\th}(\vy_{t-1}) = \vx_{t}$ and similarly $\vy_{t+1} = \phi_{b,p,\th}(\vx_t) \geq \phi_{b,p,\th}(\vx_{t-1}) = \vy_{t}$. Hence by mathematical induction the sequences $\{\vx_{t}\}_{t=0}^{\infty}$ and $\{\vy_{t}\}_{t=0}^{\infty}$ are increasing, and upper bounded by $1$. By the monotone convergence theorem the (unique) limits $\vx_{\infty}$ and $\vy_{\infty}$ exist in $[0,1]$, and from~(\ref{eq:req.sys.x}) and~(\ref{eq:req.sys.y}), satisfy
\begin{align}
\label{eq:req.sys.x.lim}
\vx_{\infty} &= \phi_{a,p,\th}(\vy_{\infty}) \,,\\
\label{eq:req.sys.y.lim}
\vy_{\infty} &= \phi_{b,p,\th}(\vx_{\infty}) \,.
\end{align}
Concretely, 
\begin{align}
\label{eq:req.sys1.lim}
\vx_{\infty} &= \phi_{a,p,\th}\left(\phi_{b,p,\th}(\vx_{\infty})\right)\,, \\
\label{eq:req.sys2.lim}
\vy_{\infty} &= \phi_{b,p,\th}\left(\phi_{a,p,\th}(\vy_{\infty})\right)\,.
\end{align}
We also note that $\vx_\infty$ and $\vy_\infty$ are non-decreasing in $p \in [0,1]$. This follows from the fact that $\vx_t$ and $\vy_t$ are non-decreasing in $p$ for every $t \geq 0$. 

We now show that there exists $\vp_f \in (0,1)$ such that $\vx_{\infty}<1$ and $\vy_{\infty}<1$ for all $p < \vp_f$, and $\vx_{\infty}=1$ and $\vy_{\infty}=1$ for all $p > \vp_f$. Let us first consider $\vx_\infty$ and $\vy_\infty$ as $p$ varies in $[0,1]$.
From~(\ref{eq:req.sys.x.lim}) and~(\ref{eq:req.sys.y.lim}), $\vx_{\infty}=1$ is equivalent to $\vy_\infty=1$ (as well as $\vx_{\infty}<1$ is equivalent to $\vy_\infty<1$). Trivially, when $p=0$, the initial probabilities $\vx_{0}=\vy_{0}=0$, yielding $\vx_{t}=\vy_{t}=0$ for every $t \geq 0$. Similarly, $\vx_{t}=\vy_{t}=1$ for every $t \geq 0$, when $p=1$. Thus there exists a value $\vp_f$ in $[0,1]$ such that $\vx_\infty$ and $\vy_\infty$ are less than $1$ for every $p < \vp_f$, and equal to $1$ for every $p > \vp_f$. 

We still have to show that the critical value $\vp_f$ is indeed in $(0,1)$. W.l.o.g. let us assume $a= \min(a,b)$ and $b=\max(a,b)$. Consider the sequences $\{\vs_{t}\}_{t=0}^\infty$ and $\{\vS_{t}\}_{t=0}^\infty$ defined by 
\beqn
\label{eq:req.sS}
\nonumber
\vs_{t} &=& \phi_{a,p,\th}(\vs_{t-1}) \,, \\
\nonumber
\vS_{t} &=& \phi_{b,p,\th}(\vS_{t-1}) \,,
\eeqn
for $t \geq 1$, where $\vs_{0}=p$, $\vS_{0}=p$. 
From the stochastic dominance on the Binomial random variable 
\beq
\pr\left(\bin\left(a,p\right) \geq \th \right) \leq \pr\left(\bin\left(b,p\right) \geq \th \right) \,.
\eeq

Now, it easily follows by mathematical induction that the sequences $\{\vs_t\}_{t=0}^\infty$ and $\{\vS_t\}_{t=0}^\infty$ represent, respectively, a lower and an upper bound on both $\{\vx_t\}_{t=0}^\infty$ and $\{\vy_t\}_{t=0}^\infty$, that is, $\vs_t \leq \vx_t, \vy_t \leq \vS_t$, for every $t \geq 0$. Analogously to the proof of the existence of $\vx_\infty$ and $\vy_\infty$, one can show that the limits $\vs_\infty:=\lim_{t \to \infty} \vs_t$ and $\vS_\infty:=\lim_{t \to \infty} \vS_t$ exist, and satisfy 
\beq
\label{eq:limits.sandwich}
\vs_\infty \leq \vx_\infty, \vy_\infty \leq \vS_\infty \,,
\eeq
and moreover
\beqn
\nonumber
\vs_{\infty} &=& \phi_{a,p,\th}(\vs_{\infty}) \,, \\
\nonumber
\vS_{\infty} &=& \phi_{b,p,\th}(\vS_{\infty}) \,.
\eeqn

By Lemma~\ref{lm:charact.pc}, given $a$, there exists the critical value $p_{a} \in (0,1)$ such that $\vs_\infty<1$ for every $p < p_a$, and $\vs_\infty=1$ for every $p > p_a$. Similarly, given $b$, there exists the critical value $p_b \in (0,1)$ such that $\vS_\infty<1$ for every $p < p_b$, and $\vS_\infty=1$ for every $p > p_b$. From the fact that $\vx_\infty,\vy_\infty, \vs_\infty,\vS_\infty$ are non-decreasing in $p$ and their relation given in~(\ref{eq:limits.sandwich}), it follows that the critical value $\vp_f$ satisfies $p_b \leq \vp_f \leq p_a$ and indeed belongs to $(0,1)$.
\end{proof}

\begin{remark}
One can prove Theorem~\ref{thm:bp.orient} for an oriented tree with periodicity bigger than two. The steps of the proof are analogous to those presented above, but instead of two sequences $\{\vx_t\}_{t=0}^{\infty}$ and $\{\vy_t\}_{t=0}^{\infty}$ we have $\ell$ sequences, where $\ell$ is equal to the periodicity of the tree.
\end{remark}

\subsection{BP on an unoriented tree $\tab$}
\label{sec:bp:unorient}

To determine the critical threshold for BP on $\tab$, we use the result of Section~\ref{sec:bp:orient} on oriented trees. 
The dynamics of bootstrap percolation process on $\tab$ is captured by knowing the states of nodes in a graph, that is, $\z_t(v) \in \{0,1\}$ and $\n_t(u) \in \{0,1\}$ for every $v \in V_a$ and $u \in V_b$ at $t \in \N_0$. Denote by $x_t$ the probability that a node of degree $a+1$ is active at time $t$, and similarly by $y_t$ the probability that a node of degree $b+1$ is active at time $t$, where $x_0=p$ and $y_0=p$.

\begin{theorem}
\label{prop2} 
The probabilities $x_{\infty}, \vx_{\infty}, y_{\infty}, \vy_{\infty}$ satisfy 
\beq
\label{eq:propx}
x_{\infty}= p + (1-p) \sum_{k=\th}^{a+1} {a+1 \choose k} \vy_{\infty}^k(1-\vy_{\infty})^{b+1-k} \,,
\eeq
and
\beq
\label{eq:propy}
y_{\infty}= p + (1-p) \sum_{k=\th}^{b+1} {b+1 \choose k} \vx_{\infty}^k(1-\vx_{\infty})^{b+1-k} \,.
\eeq
\end{theorem}

\begin{proof}
The proof consists of two parts. In Part 1, we derive the equation for the probability that a randomly selected node (w.l.o.g. of degree $b+1$) in the unoriented BP becomes active as a function of the probabilities of activation of root nodes in truncated unoriented subtrees. 
In Part 2, we relate the probability of activation of the root node in truncated unoriented subtrees to that of the truncated oriented subtrees. Combining Parts 1 and 2 establishes~(\ref{eq:propx}) and~(\ref{eq:propy}).

\textbf{Part 1.}
As previously, choose uniformly at random a node $v_0 \in \tab$ of degree $b+1$. Denote by $v_1, v_2, \dots, v_{b+1}$ the neighbors of $v_0$.
Let $T_i=\tab - (v_0,v_i)$ be a tree incident to $v_i$ obtained by removing the edge $(v_0,v_i)$ from $\tab$, see Figure~\ref{fig:treeproof}. 
In $T_i$, node $v_i$ has degree $a$, while all other nodes have degree either $b+1$ or $a+1$. 

\begin{figure}[!h]  
\centering 
\includegraphics[width=0.75\columnwidth]{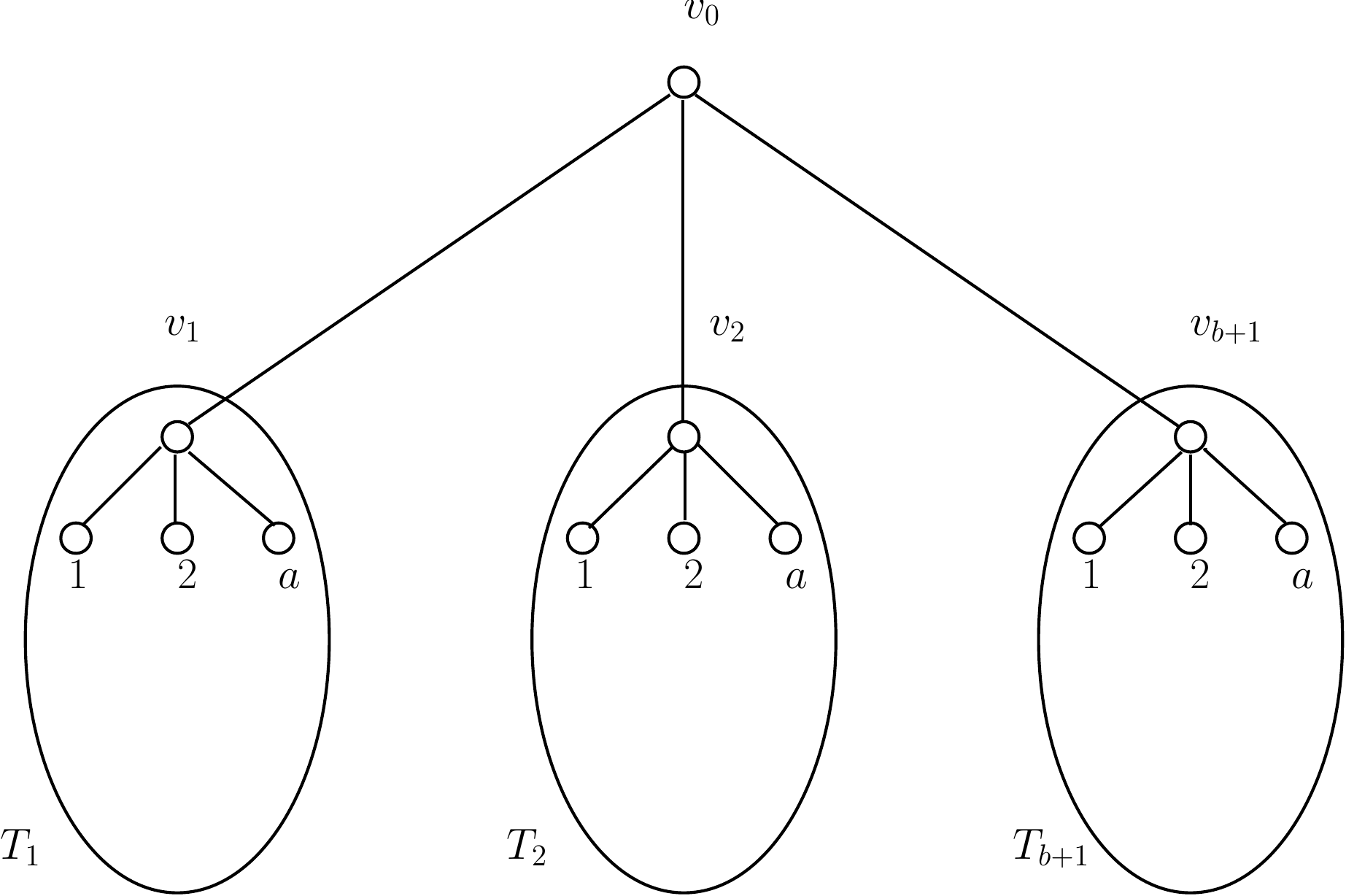}
\caption{Periodic tree with the root node $v_0$ and subtrees $T_1, T_2, \dots, T_{b+1}$ incident to the root.}
\label{fig:treeproof}
\end{figure}

Consider BP denoted by $\z^{(i)}_t$ that (starts and) runs only on $T_i$, instead of the entire tree $\tab$. 
Given $t \geq 0$, for $i=1,2,\dots,b+1$, the dynamics $\z_t^{(i)}(v_i)$ of the nodes $v_i \in T_i$ at time $t$, are i.i.d. random variables.

Now consider BP denoted by $\Xi$ that (starts and) runs on $\tab$. By symmetry and dynamics of BP, 
the process $\Xi(v_0)$ is the same in distribution for any choice of $v_0$. 
The node $v_0$ becomes active, $\Xi_{\infty}(v_0)=1$, if and only if: either (i) $\Xi_{0}(v_0)=1$, or (ii) $\sum_{i=1}^{b+1} \z_\infty^{(i)}(v_i) \geq \th$, given $\Xi_{0}(v_0)=0$. But given $\Xi_{0}(v_0)=0$, the event $\Xi_{\infty}(v_0)=0$ is equivalent to having at most $\th-1$ active neighbors, i.e., $\z_{\infty}^{(i)}(v_i)=1$. Moreover, the two BP processes: (1) $\Xi$ restricted to the tree $T_i$ given $\Xi_\infty(v_0)=0$, and (2) $\z^{(i)}_t$ (which runs on $T_i$ only) are equivalent. 
By this equivalence
\beqn
\label{eq:equivalence}
\nonumber
\pr \left( \sum_{i=1}^{b+1} \Xi_\infty(v_i) < \th : \Xi_0(v_0)=0\right) 
&=& \sum_{k=0}^{\th-1} {b+1 \choose k} \pr\left( \z_{\infty}^{(1)}(v_1)=1\right)^k \\
&& \quad \times \left(1-\pr\left( \z_{\infty}^{(1)}(v_1)=1\right) \right)^{b+1-k} \,.
\eeqn
Now, the probability that $v_0$ becomes active is given by
\begin{align}
\label{eq:KSI}
\nonumber
\pr \left( \Xi_{\infty}(v_0)=1 \right) &= \pr \left( \Xi_{0}(v_0)=1 \right) + \pr \left( \Xi_{0}(v_0)=0 \right) \\
& \quad \times \left( 1 - \pr \left( \sum_{i=1}^{b+1} \Xi_\infty(v_i) < \th : \Xi_0(v_0)=0\right) \right)\,,
\end{align}
therefore from~(\ref{eq:equivalence}),
\beqn
\label{eq:KSI}
\nonumber
\pr \left( \Xi_{\infty}(v_0)=1 \right) &=& p + (1-p) \sum_{k=\th}^{b+1} {b+1 \choose k} \pr\left( \z_{\infty}^{(1)}(v_1)=1\right)^k \\
&& \quad \times \left(1-\pr\left( \z_{\infty}^{(1)}(v_1)=1\right) \right)^{b+1-k} \,.
\eeqn
Equation~(\ref{eq:KSI}) expresses the probability that a randomly selected node $v_0$ in $\tab$ becomes active as a function of the probability of activation of the root node in a truncated unoriented subtree.

\textbf{Part 2.}
First, given the oriented edges in $\vec{T}_1$ and unoriented edges in $T_1$, it follows by
stochastic dominance that 
\beq
\label{eq:ll}
\pr \left(\vz_{\infty}^{(1)}(v_1) = 1 \right) \leq \pr \left(\z_{\infty}^{(1)}(v_1)= 1 \right) \,. 
\eeq
Next, we show that $\vz_{\infty}^{(1)}(v_1)=0$ implies $\z_{\infty}^{(1)}(v_1)=0$,
which will yield 
\beq
\label{eq:gg}
\pr \left(\vz_{\infty}^{(1)}(v_1) = 0 \right) \leq \pr \left(\z_{\infty}^{(1)}(v_1)= 0 \right) \,.
\eeq
The equivalence of activation in the directed and undirected trees will follow from~(\ref{eq:ll}) and~(\ref{eq:gg}). 

To show~(\ref{eq:gg}), 
we call a node $v$ in $\vec{T}_1$ {\it eventually-inactive} if $\vz_{\infty}^{(1)}(v)=0$ and {\it eventually-active} if $\vz_{\infty}^{(1)}(v)=1$.
Let us consider the root $v_1$ of $\vec{T}_1$. The node $v_1$ is eventually-inactive, $\vzin^{(1)}(v_1)=0$, if and only if $v_1$ is initially inactive and has at least $a-(\th-1)=a+1-\th$ eventually-inactive neighbors.
For $j \geq 0$, denote by $L_j$ the set of nodes at the level $j$ in $\vec{T}_1$. In other words, $L_0 = \{ v_1\}$, $L_1$ is the set of neighbors in $\vec{T}_1$ of the nodes in $L_0$, similarly $L_2$ is the set of neighbors in $\vec{T}_1$ of nodes in $L_1$, etc.  Every eventually-inactive node in $L_1$ has at most $\th-1$ eventually-active neighbors in $\vec{T}_1$. In other words, it has at least $b-(\th-1)=b+1-\th$ eventually-inactive neighbors from $L_2$. Given that $v_1 \in L_0$ is eventually-inactive, it follows that every eventually-inactive node in $L_1$ has at least $b+2-\th$ eventually-inactive neighbors in $\vec{T}_1$.
Similarly, every eventually-inactive node in $L_2$ has at least $a+2-\th$ eventually-inactive neighbors in $\vec{T}_1$.
Then, by mathematical induction on $j$, every eventually-inactive node in $L_j$ has at least $b+2-\th$ (respectively $a+2-\th$) eventually-inactive neighbors in $\vec{T}_1$, for odd $j$ (respectively even $j$). 

Hence $v_1$ is eventually-inactive in $\vz_t^{(1)}$ if there exists an eventually-inactive subtree $\vec{\mathcal{T}} \subseteq \vec{T}_1$, which consists of the root $v_1$, and previously recursively defined eventually-inactive nodes from $\vec{T}_1$. (Specifically, every node in the eventually-inactive three $\vec{\mathcal{T}}$ is inactive at time $t=0$.)

Now consider the unoriented BP $\z_t^{(1)}$ (on the tree $T_1$). Let $\mathcal{T}$ be unoriented copy of $\vec{\mathcal{T}}$. 
By construction of $\vec{\mathcal{T}}$, at time $t=0$, every node of $\mathcal{T}$ is inactive, and moreover has 
at least $b+2-\th$ (respectively $a+2-\th$) inactive neighbors in $T_1$, for odd $j$ (respectively even $j$).
That is, at time $t=0$, every node of $\mathcal{T}$ is inactive and has at most $\th-1$ active neighbors.
Therefore $\mathcal{T}$ is eventually-inactive under the unoriented BP $\z_t^{(1)}$, and specifically the root $v_1$ is eventually-inactive, $\zin^{(1)}(v_1)=0$. This yields~(\ref{eq:gg}), and thus
\beq
\label{eq:eqeq}
\pr \left(\vz_{\infty}^{(1)}(v_1) = 0 \right) = \pr \left(\z_{\infty}^{(1)}(v_1)= 0 \right) \,.
\eeq
Introducing $x_{\infty}:=\pr\left(\Xi_{\infty}(v_0)=1 \right)$ in~(\ref{eq:KSI}) and using~(\ref{eq:eqeq}) gives~(\ref{eq:propx}).

Analogously, one can prove the result given in~(\ref{eq:propy}) for the choice of a node $u_0$ of degree $a+1$, which concludes the proof.
\end{proof}

\begin{proposition}
\label{proposition:pf}
The percolation threshold on oriented and unoriented trees are the same:  
\beq
\vp_{f}=p_{f} \,. 
\eeq
\end{proposition}
\begin{proof}
From~(\ref{eq:propx}) and (\ref{eq:propy}) it follows that $(x_\infty,y_\infty)=(1,1)$ if and only if  $(\vx_\infty,\vy_\infty)=(1,1)$.
\end{proof}

\noindent
\textbf{Remark.} The proof readily generalizes for trees of periodicity greater than $2$ using analogous arguments. 

\section{Numerical evaluation of the critical probability $p_f$}
\label{sec:numerical}

In this section we present numerical values of $p_f$ for different values of $a=3,\dots,10$ for a non-trivial range of the threshold parameter $2 \leq \th \leq 9$ when $b=a$, $b=a+1$, and $b=a+2$, $b=2a$. Concretely, we numerically find the smallest $\vp_f$ such that the only solution of the recurrence system~(\ref{eq:req.sys1.lim}) and~(\ref{eq:req.sys2.lim}) is $(1,1)$ as justified by Theorem~\ref{thm:bp.orient}. The plots are shown in Figure~\ref{fig:numerics}. We observe that for a fixed $\th$, the critical threshold $p_f$ monotonically decreases for a fixed value of $a$ for increasing values of $b$, which agrees with expectation.

\begin{figure}[h]
\centering
\begin{tabular}{cc}
{ \includegraphics[scale=0.325]{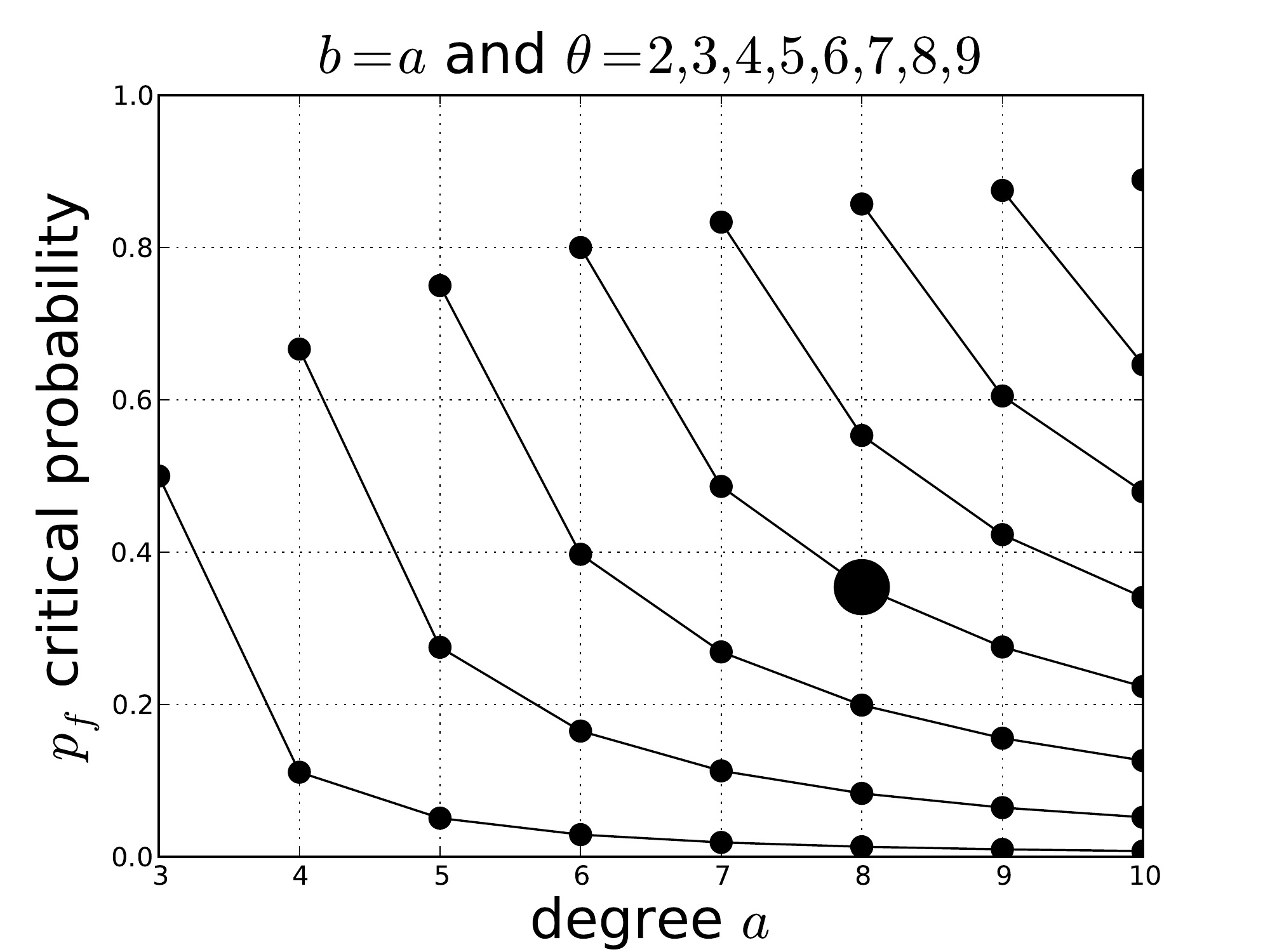} }
&
{ \includegraphics[scale=0.325]{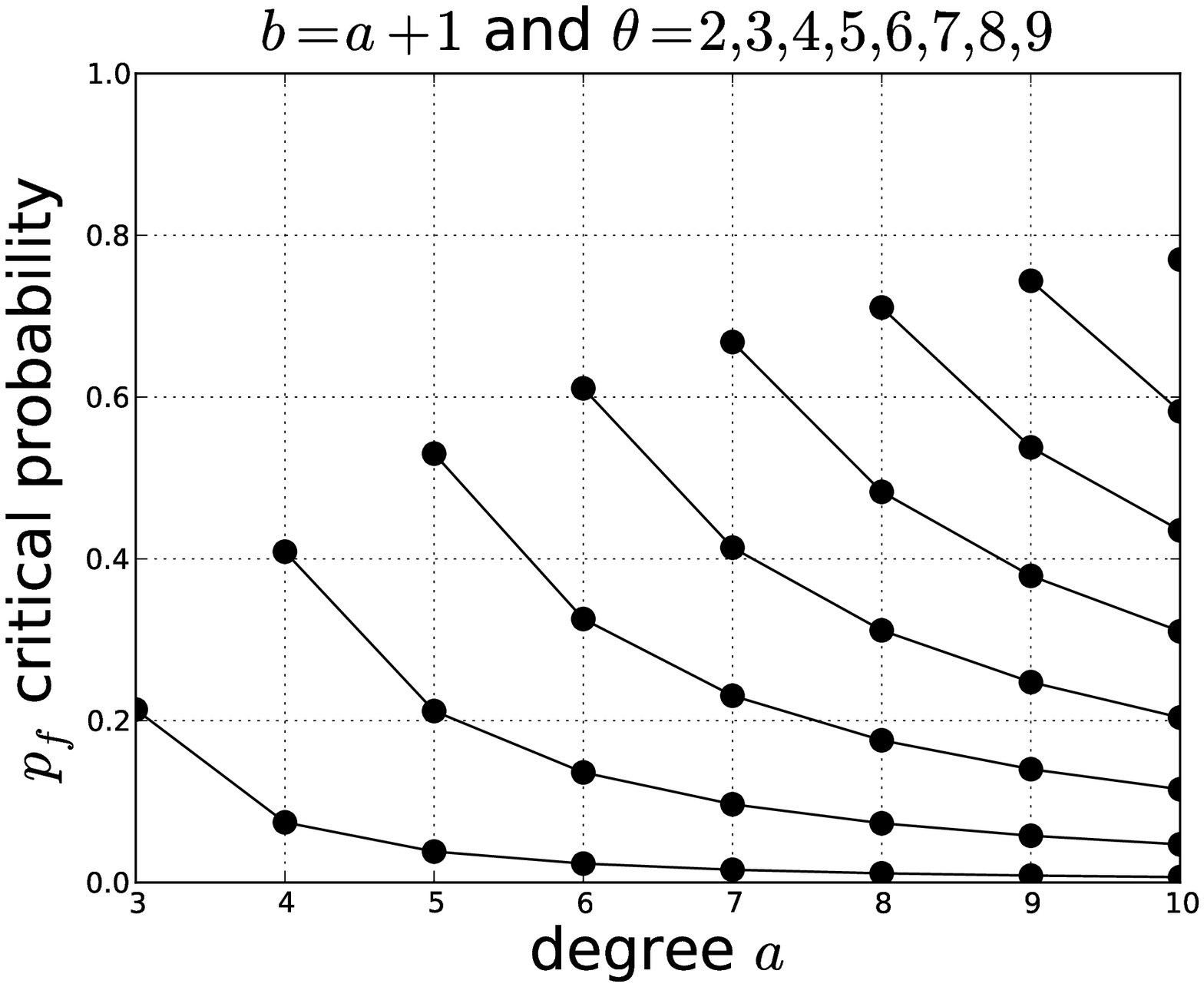} }
\\
{ \includegraphics[scale=0.325]{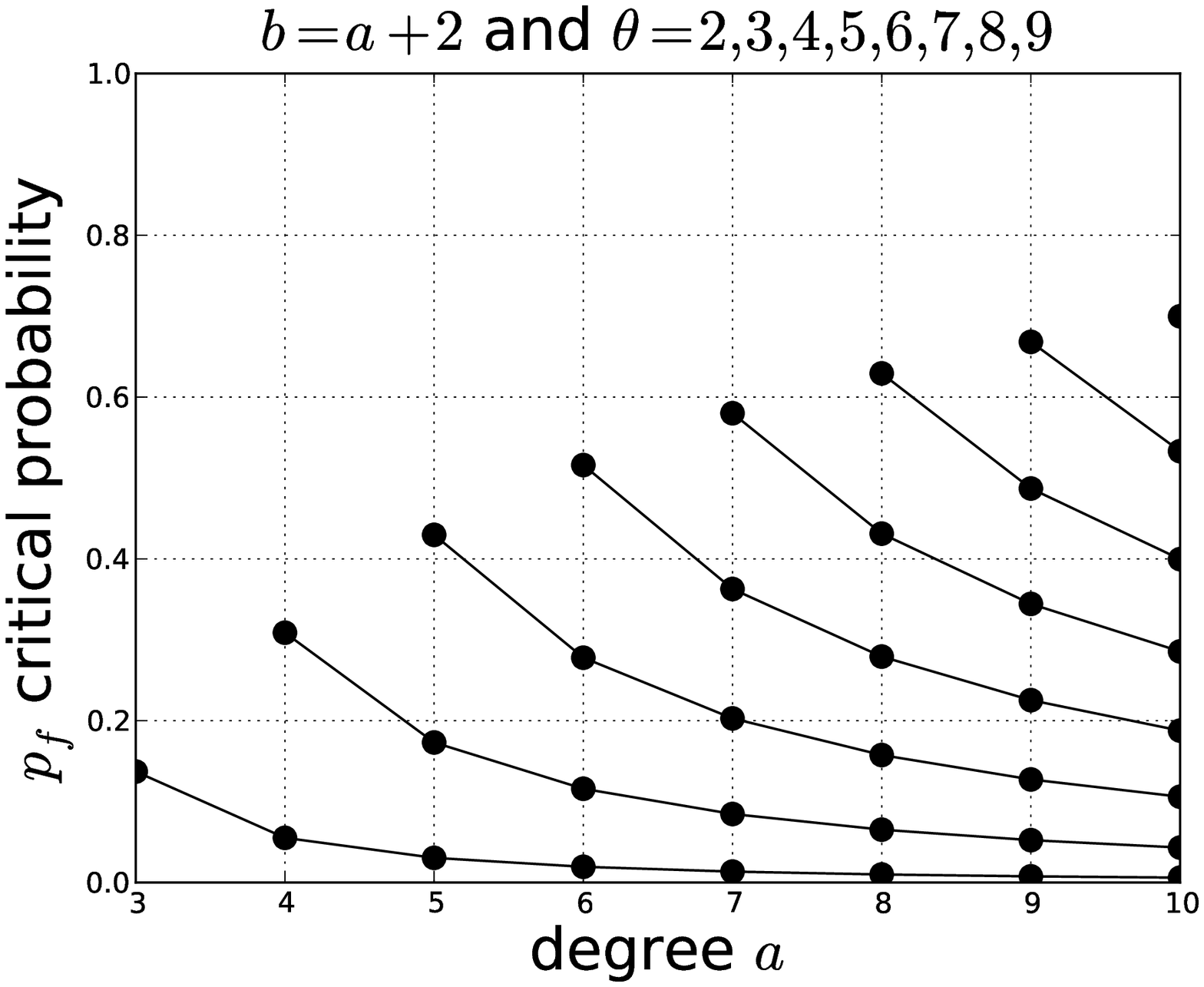} }
&
{ \includegraphics[scale=0.325]{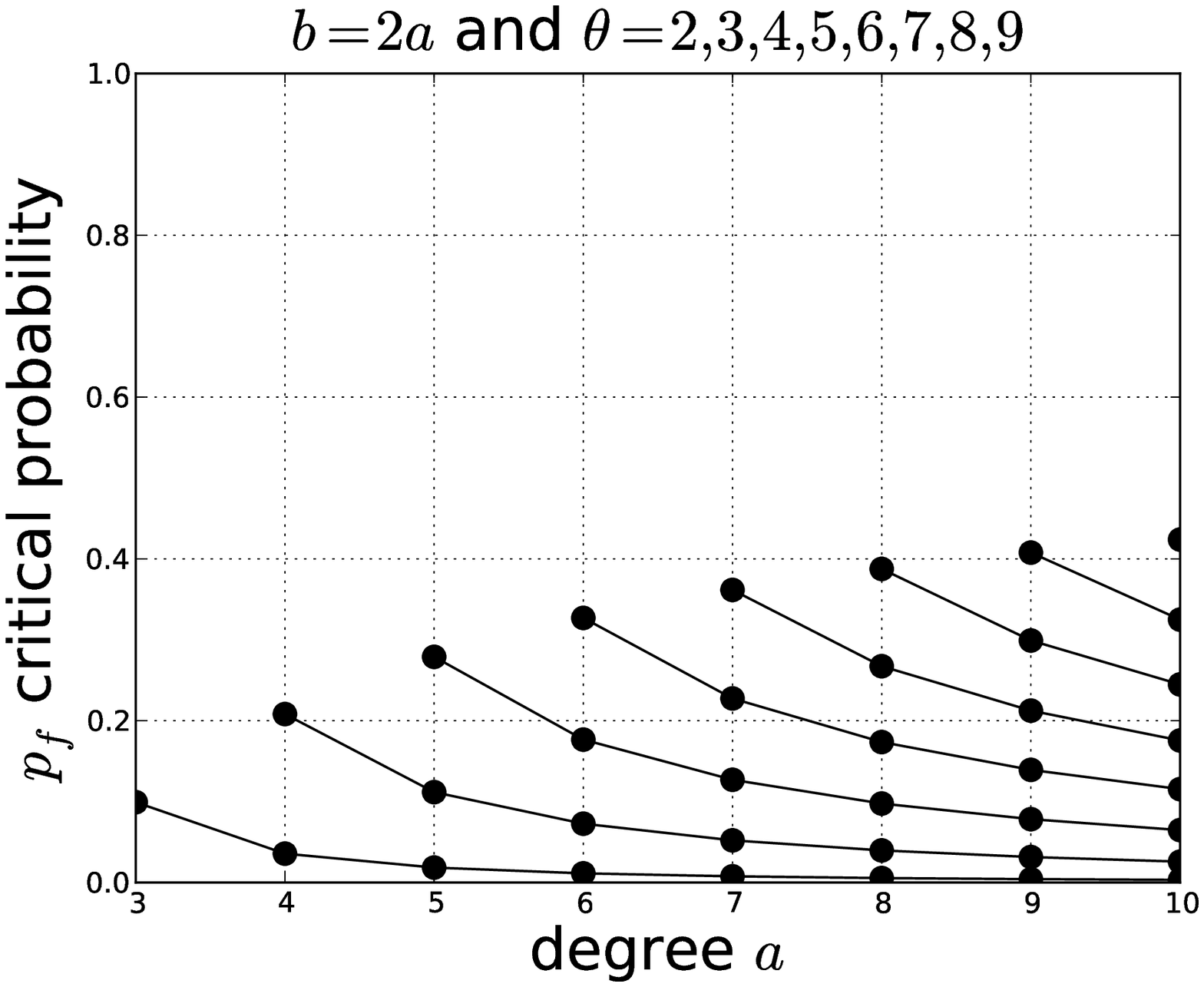} }
\end{tabular}
\caption{Numerical evaluation of the critical threshold $p_f$ for different values of $a$ and $b$ in the two-periodic tree $T_{a,b}$. 
Each curve represents $p_f$ for a given $2 \leq \th \leq 9$ and higher curves correspond to higher values of $\th$. The value of $p_f$
for $a=b=8$ and $\th=5$ corresponds to Figures~\ref{fig:curve03} and~\ref{fig:curve04}.} 
\label{fig:numerics}
\end{figure}

\section*{Acknowledgements}
This work was supported by the AFOSR grant no.~FA9550-11-1-0278 and the NIST grant no.~60NANB10D128. 

\bibliographystyle{acm}
\bibliography{boot}
\end{document}